\documentclass[a4paper]{amsart}
\usepackage{latexsym,graphicx}
\usepackage[parfill]{parskip}
\usepackage{amssymb}
\usepackage{amsthm}
\usepackage{amstext}
\usepackage{amsmath}
\usepackage{enumerate}
\usepackage{enumitem}
\usepackage{stmaryrd}
\usepackage{ mathrsfs }
\usepackage{hyperref}
\usepackage{comment}
\usepackage[active]{srcltx}
\usepackage[all]{xy}
\usepackage{mathtools}
\usepackage[a4paper]{geometry}
\usepackage{pict2e}
\usepackage{xcolor}
\usepackage{setspace}
\DeclareMathOperator\rad{rad}

\DeclareMathOperator\Hom{Hom}

\DeclareMathOperator\im{im}

\DeclareMathOperator\Ob{Ob}
\DeclareMathOperator\id{id}

\DeclareMathOperator\st{SYT}

\DeclareMathOperator\dom{dom}
\newtheorem{mythm}{Theorem}[section]
\newtheorem{mylem}[mythm]{Lemma}%[section]
\newtheorem{myprop}[mythm]{Proposition}%[section]
\newtheorem{mycor}[mythm]{Corollary}%[section]
\numberwithin{equation}{section}
\newcommand{\Z}{\mathbb{Z}}

\newcommand{\C}{\mathbb{C}}

\newcommand{\cat}{\mathcal{C}}

\newcommand{\gvsp}{\text{gVect}_\C}
\newcommand{\cgmod}{\cat\text{-gMod}}
\newcommand{\vsp}{\text{Vect}_\C}
\newcommand{\cmod}{\cat\text{-Mod}}
\newcommand{\mmod}{\text{-Mod}}

\newcommand{\inj}{\mathscr{I}}

\begin{document}
\title[Young lattice, injections between finite sets, and Koszulity]
{Incidence category of the Young lattice,\\ injections between finite sets, and Koszulity}
\author[B.~Dubsky]{Brendan Dubsky\\brendan.frisk.dubsky@math.uu.se}
%\date{\today}

\begin{abstract}
We study the quadratic quotients of the incidence category of the Young lattice defined using
the zero relations corresponding to adding two boxes to the same row, or to the same column or both. We
show that the latter quotient corresponds to the Koszul dual of the original incidence category
while the first two quotients are, in a natural way, Koszul duals of each other, hence they are in particular Koszul self-dual. Both of these two  quotients are known to be basic representatives in 
the Morita equivalence class of the category of injections between finite sets. We also present
a new, rather direct, argument establishing this Morita equivalence.
\end{abstract}

\maketitle

\section{Introduction}\label{s1}

In \cite{Po95} it was proved that the incidence algebra of a finite poset is Koszul if and
only if every open interval in this poset is Cohen-Macaulay (over the ground field). This
description is easily extendable to some locally finite infinite posets. A classical example
of an infinite Cohen-Macaulay poset is the classical Young lattice, see \cite{BS05}. The original
motivation for the present paper is to understand Koszulity of the incidence algebra of both the
Young lattice, and some of its quadratic quotients, in a direct way.

Our starting observation was that the Koszul dual of the incidence category of the Young lattice 
is a non-trivial quotient of this incidence category. In fact, in Proposition~\ref{ppp1} we show that 
the Koszul dual of the incidence category of the Young lattice is isomorphic to the quotient
of this incidence category by the ideal generated by all zero relations which correspond to 
adding two boxes to the same row or column of a Young diagram. The obvious symmetry considerations
then suggest that the quotient of incidence category of the Young lattice by the ideal generated
by all zero relations which correspond to  adding two boxes to the same row a Young diagram
should be Koszul self-dual or, more naturally, Koszul dual to the similar quotient for which
``the same row'' is changed to ``the same column''. In Theorem~\ref{selfdualthm} we prove that
this is indeed the case. Moreover, in Theorem \ref{koszthm} of Section \ref{s3} we construct explicit linear 
resolutions of simple modules for this Koszul self-dual category, which we denote by $\mathcal{C}$
in the paper.

The study of Koszul algebra originates in \cite{Pr70} and extends to the level of derived categories
in  \cite{BGS96}. In \cite{MOS09}, the Koszul duality was extended to positively graded categories
which provides us with a suitable setup for the present paper.

As it turns out, the category $\mathcal{C}$ appears naturally as the basic representative of 
the Morita equivalence class of the category whose objects are all finite sets and morphisms 
are all injections between finite sets, see for example the recent paper \cite{SS16}. 
We reprove this result in Theorem~\ref{quiverthm}, in a much more elementary way based on 
basic representation theory of finite symmetric groups. The Koszulity of $\mathcal{C}$
has also been addressed in \cite{SS16} and, even more recently, in  \cite{GL16} using completely different
approaches and without explicitly constructing projective resolutions of simple modules. 

As the category of injections between finite sets seems to be a more natural object than the
Young lattice, we try to present our results from the perspective of the former category.
Therefore, let us make some more detailed comments on this category.

Consider the category $\inj$ with objects $\underline{n}$ for $n\in\Z_{\ge 0}$, where 
$\underline{0}=\varnothing$ and otherwise $\underline{n}=\{ 1,\dots, n\}$, and whose morphisms 
are injections between these sets. The category $\inj$ is a skeletal subcategory 
of the category of injections between finite sets. Clearly,
each endomorphism semigroup of $\inj$ is isomorphic to some symmetric group $S_n$. 
There is a natural $\Z_{\ge 0}$-grading on $\inj$ (and hence on its linearization 
$\C\inj$) given by letting the degree of an injection $i\in \inj(\underline{n},\underline{n+m})$ be $m$.

It is also worth to note that the representation theory of the category of 
injections between finite sets has been studied in 
\cite{CEF15} and \cite{CEFN14}, where, in particular, several examples of how modules over 
this category arise in applications can be found.  The Gabriel quiver (without relations) of the 
linearized category $\C\inj$ was described already in \cite{Br11}, and later in the much 
more general setting of rectangular monoids in \cite{MS12}. Their approach
was used by in \cite{St16} to describe the Gabriel quiver of the category of surjections 
between finite sets as well as some related categories. 

The connection of our results to the results of \cite{SS16} was  brought to the author's attention
after publication of the original version of this paper on the arxiv.
\vspace{5mm}

\subsection*{Acknowledgements}
The author is very grateful to his advisor Volodymyr Mazorchuk for many valuable insights and patient guidance during the project that resulted in the present paper. 
The author would also like to thank Steven Sam for information about the results of  \cite{SS16}.

\section{Notation}
In this section we collect notation that will be used throughout the paper. 

\subsection{Miscellaneous}
We write $\Z_{\ge 0}=\{0,1,2,\dots\}$ for the set of non-negative integers.

For a function $f$, we by $\dom(f)$ denote the domain of $f$, and by $\im(f)$ the image of $f$. 

Any tensor products not specified via subscripts will be assumed to be taken over $\C$.

For a finite-dimensional complex vector space $V$, we denote by $V^*=\Hom(V,\C)$ the usual dual vector space. If $\{v_1,\dots,v_n\}$ is a basis for $V$, then the $v_i^*\in V^*$ defined by 
\begin{equation*}
v_i^*(v_j)=\begin{cases}
        1, & \mbox{if } i=j,\\
        0, & \mbox{ otherwise,}
        \end{cases}
\end{equation*}
constitute a basis for $V^*$. 

\subsection{The symmetric group}

For a finite set $A$ we denote by $S_A$ the symmetric group of permutations of the elements of $A$. In the special case $A=\underline{n}$ for $n\in\\Z_{\ge 0}$, we simply write $S_n$ instead of $S_{\underline{n}}$. For a subset $B\subset A$ we will often identify $S_B$ with the subgroup of $S_A$ consisting of all $\pi\in S_A$ such that $\pi_{|A\backslash B}=\id$, the identity function. A similar identification of group algebras will be done for $\C[S_B]$ as a subalgebra of $\C[S_A]$. 

By the (possibly decorated) letters $\lambda,\mu,\nu$ we will denote Young diagrams, i.e. arrangements of left-justified rows of nodes (drawn as square ``boxes''), where the length of the rows are weakly decreasing, and $\lambda\vdash n\in\Z_{\ge 0}$ means that $\lambda$ consists of $n$ nodes in total, with the convention that $\varnothing\vdash 0$. We write $\lambda=(n_1,n_2,\dots,n_k)$ if $\lambda$ has $n_i$ nodes in its $i$:th row (and no nodes in rows $>k$). By $\mu^T$ we mean the transposed diagram, which has each row of $\mu$ as a column, and vice versa. 

If it is possible to adjoin a new node to a Young diagram and thereby again obtain a Young diagram, we say that this node is \emph{addable}. Whenever we speak of adding a node to a Young diagram, the node is implicitly understood to be addable. We will call a quadruple $(\lambda_1,\lambda_2,\lambda_3,\lambda_4)$ a \emph{diamond} if $\lambda_2$ and $\lambda_3$ are different and can be obtained from $\lambda_1$ by adding one node, and $\lambda_4$ can be obtained from $\lambda_2$ and $\lambda_3$ by adding one node. We write $\mu\rightarrow\lambda$ if $\lambda$ can be obtained from $\mu$ by adding nodes, and $\mu\xrightarrow{2}\lambda$ if $\lambda$ can obtained from $\mu$ by adding nodes, at least two of which to the same column.

We use $t_\lambda$ to denote a Young tableau of shape $\lambda\vdash n$, i.e. an object obtained by entering the numbers $1,\dots, n$ into the nodes of $\lambda$. Tableaux whose entries increase along rows and columns are called standard, and by $\st(\lambda)$ we denote the set of all standard tableaux of shape $\lambda$. By $t^1_\mu$ we denote the element of $\st(\mu)$ obtained by by entering the numbers $1,\dots,n$ into $\mu$ from top to bottom and from left to right (actually, any fixed element of $\st(\mu)$ would do for our purposes).

By $S^\lambda$ we denote the Specht module corresponding to $\lambda$. See Subsection \ref{s22} for a more detailed discussion. 

\subsection{Modules over categories and categories of modules}
\label{ss23}

All vector spaces considered in the present paper will be complex, and, unless otherwise specified, linear will always mean $\C$-linear. We denote by $\vsp$ the category of vector spaces and linear maps, and by $\gvsp$ the category of graded vector spaces and graded linear maps. 

For a small category $\mathcal{E}$, we by $\C\mathcal{E}$ denote the \emph{$\C$-linearization} of $\mathcal{E}$. This is the $\C$-linear category that has the same objects as $\mathcal{E}$, but with morphism space $\C\mathcal{E}(X,Y)$ being the vector space with basis $\mathcal{E}(X,Y)$, and where composition of morphisms is defined by bilinearly extending the composition in $\mathcal{E}$. 

We by $\rad(\C\mathcal{E})$ denote the radical of the category $\C\mathcal{E}$, which is the two-sided ideal of $\C\mathcal{E}$ defined by
\begin{equation*}
\rad(\C\mathcal{E})(X,Y)=\{h\in\C\mathcal{E}(X,Y)|1_X-g\circ h\text{ is invertible for all }g\in\C\mathcal{E}(Y,X)\}
\end{equation*}
We note in passing that, as in the case of algebras, the radical may also be defined as the intersection of the annihilators of all simple modules. 

A (left) $\mathcal{E}$-module is a (covariant) $\C$-linear functor 
\begin{equation*}
F:\C\mathcal{E}\rightarrow\vsp. 
\end{equation*}
By $\mathcal{E}\mmod$ we denote the category of $\mathcal{E}$-modules and natural transformations between these. 

If $\mathcal{E}$ is $\mathbb{Z}$-graded, then a correspondingly graded $\mathcal{E}$-module is a $\C$-linear functor 
\begin{equation*}
F:\C\mathcal{E}\rightarrow\gvsp 
\end{equation*}
which preserves the degrees of morphisms. We denote by $\mathcal{E}\text{-gMod}$ the category of graded $\mathcal{E}$-modules and degree-preserving natural transformations between these.

We also define a right $\mathcal{E}$-module to be a contravariant $\C$-linear functor 
\begin{equation*}
G:\C\mathcal{E}\rightarrow\vsp. 
\end{equation*}
Along the same lines one defines right graded $\mathcal{E}$-modules, as well as the categories $\text{Mod-}\mathcal{E}$ and $\text{gMod-}\mathcal{E}$ of ungraded and graded right $\mathcal{E}$-modules respectively. 

We will often identity a left $\mathcal{E}$-module $M$ with the vector space 
\begin{equation*}
\bigoplus_{X\in\mathcal{E}}M(X)
\end{equation*}
together with the $\mathcal{E}$-action given by 
\begin{equation*}
f\cdot v=M(f)(v),
\end{equation*}
where the domain of $f\in\mathcal{E}(X,Y)$ is extended to the vector space $M$ by setting $M(f)(M(Z))=0$ for summands $M(Z)\ne M(X)$. A similar identification may be done for right $\mathcal{E}$-modules. 

For a left ideal $\mathcal{J}\subset \mathcal{E}$ we obtain an obvious ideal $\C\mathcal{J}\subset\C\mathcal{E}$, and we may define the (possibly graded) left $\mathcal{E}$-module
\begin{equation*}
^\mathcal{E}\mathcal{J}=\bigoplus_{X\in\mathcal{E}}\C\mathcal{J}(X,\_),
\end{equation*}
and similarly a right ideal $\mathcal{J}\subset \mathcal{E}$ lets us define the right $\mathcal{E}$-module
\begin{equation*}
\mathcal{J}^\mathcal{E}=\bigoplus_{X\in\mathcal{E}}\C\mathcal{J}(\_,X).
\end{equation*}

Let $n$ range over $\Z$ and $X$ over the objects of $\cat$. We define the \emph{degree shift} functors 
\begin{equation*}
\langle n\rangle:\cgmod\rightarrow \cgmod
\end{equation*}
on objects $M\in\cgmod$ by 
\begin{equation*}
M\langle n\rangle(X)_m=M(X)_{n+m}
\end{equation*}
for all $m\in\Z$ (and in the obvious way on morphisms).

\section{Quiver of the category of injections between finite sets}\label{s2}

\subsection{Description of the quiver of $\C\inj$}
Let $\mathcal{C}'$ be (the incidence category of) the Young lattice, i. e. the poset consisting of the Young diagrams, ordered by $\mu>\lambda$ if $\mu\rightarrow\lambda$. 

\begin{figure}
\includegraphics{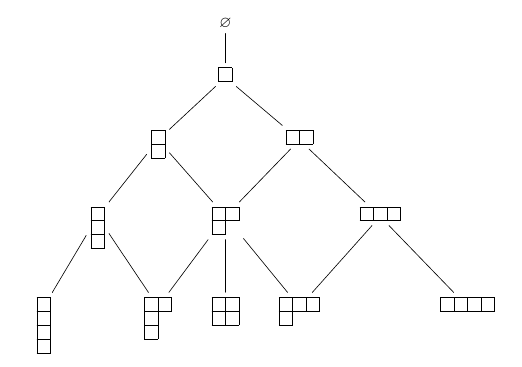}
\begin{comment}
\xymatrix{
&&&\varnothing\ar@{-}[d]&&&\\
&&&\ydiagram{1}\ar@{-}[dl]\ar@{-}[dr]&&&\\
&&\ydiagram{1,1}\ar@{-}[dl]\ar@{-}[dr]&&\ydiagram{2}\ar@{-}[dl]\ar@{-}[dr]&&\\
&\ydiagram{1,1,1}\ar@{-}[dl]\ar@{-}[dr]&&\ydiagram{2,1}\ar@{-}[d]\ar@{-}[dl]\ar@{-}[dr]&&\ydiagram{3}\ar@{-}[dl]\ar@{-}[dr]&\\
\ydiagram{1,1,1,1}&&\ydiagram{2,1,1}&\ydiagram{2,2}&\ydiagram{3,1}&&\ydiagram{4}
}
\end{comment}
\caption{The Young lattice truncated after diagrams of size four.}
\end{figure}

Let $Q$ be the (quiver given by the) Hasse diagram of $\mathcal{C}'$, i.e. the set of nodes consists of the objects of $\mathcal{C}'$, with an arrow going from node $\mu$ to node $\lambda$ if and only if $\lambda$ is obtained from $\mu$ by adding a node.

Finally, consider the ideal $\mathcal{I}=\bigcup_{\mu\xrightarrow{2}\lambda}\mathcal{C}'(\mu,\lambda)$, and let $\mathcal{C}=\mathcal{C}'/\mathcal{I}$.

Our first main result reads as follows.

\begin{mythm}
\label{quiverthm}
The Gabriel quiver of $\C\inj$ is $Q$, and $\C\inj$ is Morita equivalent to $\C\mathcal{C}$, i.e. $\C\inj\text{-Mod}\cong \C\cmod$.
\end{mythm}

Note that $\C\mathcal{C}'$ and hence $\C\cat$ are positively graded (in the sense of \cite[Definition 1]{MOS09}), by the length of the corresponding paths in $Q$. From the proof of Theorem \ref{quiverthm}, it is easily seen that the following corollary holds as well.

\begin{mycor}
\label{greqcor}
We have the equivalence of categories $\C\inj\text{-gMod}\cong\C\cgmod$.
\end{mycor}

\subsection{The symmetric group and its simple representations}\label{s22}
According to classical results on the ordinary representation theory of finite groups (cf. e.g. \cite[Section 3.1]{JK81}), there is a decomposition into two-sided ideals of the group algebra of $S_n$
\begin{equation*}
\C[S_n]= \bigoplus_{\mu\vdash n}\C[S_n] E_\mu,
\end{equation*}
where $E_\mu\in \C[S_n]$ is the \emph{centrally primitive idempotent} corresponding to $\mu$. As the terminology suggests, $E_\mu$ is an idempotent and belongs to the center of $\C[S_n]$, so that in particular it acts as the identity on $\C[S_n] E_\mu$. For an explicit formula for $E_\mu$, cf. \cite[formula (1), p. 188]{JK81}.\footnote{The alternative formula given as formula (2) in the same source appears to be erroneous, due to Young symmetrizers not being orthogonal in general.}

We also have, in turn, decompositions of left $\C[S_n]$-modules 
\begin{equation*}
\C[S_n] E_\mu\cong \bigoplus_{t_\mu\in\st(\mu)} \C[S_n] e_{t_\mu},
\end{equation*}
where $e_{t_\mu}$ is the Young symmetrizer corresponding to $t_\mu$, and $\C[S_n] e_{t_\mu}\cong S^\mu$ is simple. 

\subsection{The symmetric group, $\C\inj$ and idempotents}\label{s24}
We will study $\C\inj$ using the representation theory of $S_n$. Note that we have obvious isomorphisms of algebras
\begin{equation*}
\gamma_n:\C[S_n]\xrightarrow{\sim}\C\inj(\underline{n},\underline{n}).
\end{equation*}
In particular we have that $\gamma_n$ induces a left (right) $\C[S_n]$-module structure on any left (right) $\C\inj$-module $M$, namely the module 
\begin{equation*}
\bigoplus_{X\in\C\inj}M(X) 
\end{equation*}
with left (right) multiplication by $s\in\C[S_n]$ given by applying $M(\gamma_n(s))$. Each $\C\inj(\underline{n},\underline{n+m})$ then inherits a $\C[S_{n+m}]$-$\C[S_{n}]$-bimodule structure from $^{\C\inj}{\C\inj}$ and $\C\inj^{\C\inj}$. 

One may heuristically view injections in $\inj(\underline{n},\underline{n+m})$ as permutations in $S_{n+m}$, but where we do not care about where the elements of $\underline{n+m}\backslash\underline{n}$ are sent. This is the same as sending the elements $\underline{n+m}\backslash\underline{n}$ ``everywhere possible at once''. Therefore we also have the isomorphisms of $\C [S_{n+m}]$-$\C [S_n]$-bimodules
\begin{align*}
\iota_{n,m}:M_{n,m}:=\langle \sum_{\pi\in S_{\underline{n+m}\backslash\underline{n}}} \sigma\pi | \sigma\in S_{n+m}\rangle &\xrightarrow{\sim} \C\inj(\underline{n},\underline{n+m})\\
\sum_{\pi\in S_{\underline{n+m}\backslash\underline{n}}} \sigma\pi&\mapsto \sigma_{|\underline{n}}.
%\C\inj(\underline{n},\underline{n+m})&\xrightarrow{\sim} \langle \sum_{\pi\in S_{\underline{n+m}\backslash\underline{n}}} \sigma\pi | \sigma\in S_n\rangle=:M\subset \C S_{n+m}\\
%\sigma&\mapsto \sum_{\pi\in S_{\underline{n+m}\backslash\underline{n}}} \sigma.
\end{align*}
Since $M_{n,m}\subset \C[S_{n+m}]$, we obtain an embedding $\iota_{n,m}^{-1}:\C\inj(\underline{n},\underline{n+m})\rightarrow \C[S_{n+m}]$. 

We will study $\C\inj$ via its idempotents, which correspond to ones in the symmetric group algebra: Write 
\begin{equation*}
e_{t_\mu}'=\gamma_n(e_{t_\mu})
\end{equation*}
and 
\begin{equation*}
E'_{\mu}=\gamma_n(E_\mu), 
\end{equation*}
with $\mu\vdash n$ and $t_\mu\in\st(\mu)$. Every primitive idempotent $e\in\C\inj$ satisfying $\C\inj(\underline{n},\underline{n+1})e\ne 0$ must lie in $\C\inj(\underline{n},\underline{n})\cong \C [S_n]$ (cf. \cite[p. 73 and Theorem 6.3.2]{Br11} for details), and is therefore in fact of the form $e_{t_\mu}'$. Similarly, the primitive idempotents $e\in\C\inj$ satisfying $e\C\inj(\underline{n},\underline{n+1})\ne 0$ are precisely the $e'_{t_\lambda}$ with $\lambda\vdash n+1$.

\subsection{Simples and idecomposable projectives in $\C\inj\text{-Mod}$}\label{s23}
From the classification of primitive idempotents of $\C\inj$ given in the previous subsection, it is immediate that the indecomposable projective $\C\inj$-modules are of the form
\begin{equation*}
P_\mu= {^{\C\inj}(\C\inj e'_{t^1_\mu})}.
\end{equation*} 

The simple top, denoted $L_\mu$, of $P_\mu$ satisfies $\C\inj(\underline{n},\underline{n+m}) L_\mu=0$ for $m>0$ and $L_\mu\cong S^\mu$ as $\C[S_n]$-modules. 

Observe that the same notation will be used in a different way in Section \ref{s3}. 

\begin{mylem}
\label{projeqlem}
Let $\nu=(m)$, and view $\C[S_n]\otimes \C[S_m]$ as a subalgebra of $\C[S_{n+m}]$ in the obvious way. We have the $\C[S_{n+m}]$-module isomorphism
\begin{equation*}
\C[S_{n+m}]P_\mu\cong \C[S_{n+m}]\otimes_{\C[S_n]\otimes \C[S_m]}(S^\mu\otimes S^\nu).
\end{equation*}
\end{mylem}
\begin{proof}
We compute
\begin{equation*}
%\label{projeq}
\begin{aligned}
\C[S_{n+m}]P_\mu&=\C[S_{n+m}]{^{\C\inj}(\C\inj e'_{t^1_\mu})}\\
&\cong \C\inj(\underline{n},\underline{n+m})e'_{t^1_\mu}\\
&\cong \C\inj(\underline{n},\underline{n+m})\otimes_{\C[S_n]}\C\inj(\underline{n},\underline{n})e'_{t^1_\mu}\\
&\cong \C\inj(\underline{n},\underline{n+m})\otimes_{\C[S_n]} S^\mu\\
&\cong M_{n,m}\otimes_{\C[S_n]} S^\mu\xrightarrow{\sim} \C[S_{n+m}]\otimes_{\C[S_n]\otimes \C[S_m]}(S^\mu\otimes S^\nu)\\
&{\hspace{36 pt}}x\otimes_{\C[S_n]} s_\mu\mapsto x\otimes_{\C[S_n]\otimes \C[S_m]}(s_\mu\otimes s_\nu),
\end{aligned}
\end{equation*}
where $S^\nu$ is the trivial module and $s_\nu\in S^\nu$ is some fixed basis element. 
\end{proof}

\subsection{Proof of Theorem \ref{quiverthm}}

\begin{proof}[Proof of Theorem \ref{quiverthm}]
We want need to find the nodes, arrows and relations of the quiver of $\C\inj$. The first two are a result of Brimacombe (cf. \cite[Theorem 8.1.2]{Br11}), whose proof we outline for the convenience of the reader. 

The nodes are indexed by the primitive idempotents $e_{t^1_\mu}'$, which are in turn indexed by the Young diagrams $\mu$. 

The arrows of the Gabriel quiver of $\C\inj$ correspond to a basis of $\rad(\C\inj)/\rad^2(\C\inj)$. It is obvious from our definition of the radical (cf. Subsection \ref{ss23}) that $\rad(\C\inj)=(\C\inj)_{\ge 1}$, and therefore $\rad(\C\inj)/\rad^2(\C\inj)=(\C\inj)_1$. Thus every arrow of the quiver corresponds to a basis element of some $\C\inj(\underline{n},\underline{n+1})$. 

Let for the remainder of this proof $\mu\vdash n$ be fixed. 

We have an obvious $\C [S_{n+1}]$-$\C [S_n]$-bimodule isomorphism 
\begin{equation*}
\C\inj(\underline{n},\underline{n+1})\cong \C[S_{n+1}].
\end{equation*}
There is also the $\C[S_{n+1}]$-module isomorphism 
\begin{align*}
\C [S_{n+1}]e_{t^1_\mu}&\cong \C [S_{n+1}]\otimes_{\C [S_n]} \C [S_{n}]e_{t^1_\mu}\\
ae_{t^1_\mu}&\leftrightarrow a\otimes e_{t^1_\mu}.
\end{align*}
Finally, for $\lambda\vdash n+1$, we have the $e_{t^1_\lambda}\C [S_{n+1}] e_{t^1_\lambda}$-module isomorphism
\begin{align*}
\Hom_{\C [S_{n+1}]}(\C [S_{n+1}]e_{t^1_\lambda}, \C [S_{n+1}]\otimes_{\C [S_{n}]}\C [S_n] e_{t^1_\mu})&\xrightarrow{\sim} e_{t^1_\lambda}(\C [S_{n+1}]\otimes_{\C [S_n]} \C [S_{n}]e_{t^1_\mu})\\
\phi&\mapsto \phi(e_{t^1_\lambda}).
\end{align*}
Altogether we see that the nodes of the Gabriel quiver of $\C\inj$ may be identified with partitions of non-negative integers, with the number of arrows from $\mu$ to $\lambda$ given by 
\begin{align*}
\dim(e'_{t^1_\lambda}\C\inj(\underline{n},\underline{n+1})e'_{t^1_\mu})&=\dim(e_{t^1_\lambda}\C[S_{n+1}]e_{t^1_\mu})\\&=\dim(e_{t^1_\lambda}(\C [S_{n+1}]\otimes_{\C [S_n]} \C [S_{n}]e_{t^1_\mu}))\\&=\dim(\Hom_{\C [S_{n+1}]}(\C [S_{n+1}]e_{t^1_\lambda}, \C [S_{n+1}]\otimes_{\C [S_n]} \C [S_{n}]e_{t^1_\mu})\\&=\dim(\Hom_{\C [S_{n+1}]}(S^\lambda, \C [S_{n+1}]\otimes_{\C [S_n]} S^\mu).
\end{align*}
This number is by the Branching rule $1$ if $\lambda$ is obtainable from $\mu$ by adding a node, and $0$ otherwise (cf. \cite[Theorem 2.8.3 p. 77]{Sa01}). 

We now turn our attention to the relations of the quiver. Let $\lambda\vdash n+m$ be a Young diagram obtained from $\mu\vdash n$ by adding nodes. We begin by computing the number of $L_\lambda$-subquotients in $P_\mu$. Left multiplication with $E'_\lambda$ is the same as projection onto the $L_\lambda$-subspaces, so the maximal subquotient of $P_\mu$ isomorphic to a sum of copies of $L_\lambda$ is given by ${E'}_\lambda P_\mu$. We claim that we have
\begin{equation}
\label{projlvls}
E'_\lambda P_\mu \cong \begin{cases}
        0, & \mbox{if } \mu\xrightarrow{2}\lambda\\
        L_\lambda, & \mbox{ otherwise.}
        \end{cases} 
\end{equation}
Indeed, the number of simple subquotients $L_\lambda$ in the $\C\inj$-module $P_\mu$ equals the number of simple subquotients $S^\lambda$ in the $\C[S_{n+m}]$-module $P_\mu$. The projection of $P_\mu$ onto its non-zero $\C[S_{n+m}]$-summands is given by $\C[S_{n+m}]P_\mu$. Let $\nu=(m)$. By Lemma \ref{projeqlem}, we have the isomorphism of $\C[S_{n+m}]$-modules
\begin{equation*}
\C[S_{n+m}]P_\mu\cong \C[S_{n+m}]\otimes_{\C[S_n]\otimes \C[S_m]}(S^\mu\otimes S^\nu)\cong \bigoplus_{\lambda'\vdash n+m} c^{\lambda'}_{\mu,\nu}S^{\lambda'},
\end{equation*}
where the \emph{Littlewood-Richardson coefficients} $c^{\lambda'}_{\mu,\nu}$ may be calculated by the Littlewood-Richardson rule (cf. Theorem 4.9.4 of \cite{Sa01}). The rule tells us that $c^{\lambda'}_{\mu,\nu}$ equals the number of semistandard tableaux $t$ which have shape $\lambda'\backslash\mu$, content $\nu$ and whose reverse row word is a lattice permutation. But since the content diagram $\nu$ is a single row, every entry of $t$ is a 1. Therefore only the semistandardness condition may be unsatisfied, and this will happen precisely when $\lambda\backslash \mu$ has at least two nodes in the same column (in which case $c^{\lambda'}_{\mu,\nu}=0$), and when the condition is satisfied, $t$ is uniquely determined (in which case $c^{\lambda'}_{\mu,\nu}=1$). The claim follows. 

Let next $\nu\vdash n+m$ be obtained from $\mu\vdash n$, and $\lambda\vdash n+m+l$ be obtained from $\nu$ by adding nodes, no two of which to the same column. We will show that the submodule of $P_\mu$ generated by $L_\nu$ contains $L_\lambda$ as a subquotient. This will prove that in the quiver of $\C\inj$, the path from $\mu$, through $\nu$, to $\lambda$ is non-zero. It suffices to show that
\begin{equation*}
{^{\C\inj}(\C\inj E'_\nu\C\inj(\underline{n},\underline{n+m})e'_{t^1_\mu})}
\end{equation*}
contains the subquotient $L_\lambda$. 

Clearly we have the $\C[S_{n+m+l}]$-module isomorphism
\begin{align*}
\C\inj(\underline{n+m},\underline{n+m+l})\otimes_{\C[S_{n+m}]}&\C\inj(\underline{n},\underline{n+m})\\&\xrightarrow{\sim}\C\inj(\underline{n+m},\underline{n+m+l})\C\inj(\underline{n},\underline{n+m})\\
a\otimes_{\C[S_{n+m}]}b&\mapsto ab.
\end{align*}
This restricts to the isomorphism of submodules
\begin{align*}
&E'_\lambda\C\inj(\underline{n+m},\underline{n+m+l})E'_\nu\C\inj(\underline{n},\underline{n+m})\\&\cong E'_\lambda\C\inj(\underline{n+m},\underline{n+m+l})\otimes_{\C[S_{n+m}]}E'_\nu\C\inj(\underline{n},\underline{n+m}).
\end{align*}
Using this and the isomorphism \eqref{projlvls}, we get the $\C[S_{n+m+l}]$-module isomorphisms
\begin{equation}
\label{longcalc}
\begin{aligned}
&E'_\lambda {^{\C\inj}(\C\inj E'_\nu\C\inj(\underline{n},\underline{n+m})e'_{t^1_\mu})}\\
&\cong E'_\lambda\C\inj(\underline{n+m},\underline{n+m+l})E'_\nu\C\inj(\underline{n},\underline{n+m})e'_{t^1_\mu}\\
&\cong E'_\lambda\C\inj(\underline{n+m},\underline{n+m+l})\otimes_{\C[S_{n+m}]}E'_\nu\C\inj(\underline{n},\underline{n+m})e'_{t^1_\mu}\\
&\cong E'_\lambda\C\inj(\underline{n+m},\underline{n+m+l})\otimes_{\C[S_{n+m}]}E'_\nu P_\mu\\
&\cong E'_\lambda\C\inj(\underline{n+m},\underline{n+m+l})\otimes_{\C[S_{n+m}]}L_\nu\\
&\cong E'_\lambda P_\nu\\
&\cong L_\lambda\ne 0.
\end{aligned}
\end{equation}

It remains to check that commutativity of different paths between the same vertices holds in our quiver. It suffices to check this for paths of length two, i. e. check that we may fix non-zero morphisms 
\begin{equation*}
\phi_{\lambda,\mu}:P_\lambda\rightarrow P_\mu
\end{equation*}
for all diagrams $\lambda$ and $\mu$, with $\lambda$ obtained from $\mu$ by adding a node, so that for every diamond $(\lambda_1,\lambda_2,\lambda_3,\lambda_4)$ we have
\begin{equation}
\label{comm}
\phi_{\lambda_2,\lambda_1}\circ\phi_{\lambda_4,\lambda_2}=\phi_{\lambda_3,\lambda_1}\circ\phi_{\lambda_4,\lambda_3}.
\end{equation}

First note that no matter how we pick the morphisms $\phi_{\lambda,\mu}$, we will have
\begin{equation*}
\phi_{\lambda_2,\lambda_1}\circ\phi_{\lambda_4,\lambda_2}=c\phi_{\lambda_3,\lambda_1}\circ\phi_{\lambda_4,\lambda_3},
\end{equation*}
for some non-zero $c\in\C$ (depending on $\lambda$ and $\mu$). This follows immediately from Schur's lemma together with the calculation \eqref{longcalc}, the latter showing that neither the LHS nor the RHS of \eqref{comm} is zero. 

Now define the $\phi_{\lambda,\mu}$ by extending to morphisms the maps
\begin{align*}
\phi_{\lambda,\mu}:P_\lambda&\rightarrow P_\mu\\
v_\lambda&\mapsto v_{\lambda,\mu},
\end{align*}
where the $v_\lambda$ and $v_{\lambda,\mu}$ are defined according to the following algorithm.

Let $k\in\Z_{>1}$.
\begin{description}
\item[Step 1]
For every $n\in\Z_{>0}$ and every diagram $\lambda\vdash n$, fix any non-zero vector $v_\lambda$ in $P_\lambda(\underline{n})$. Let $\mu_2=(1)$ be the diagram that consists of one node. 

\item[Step k]
For every $\lambda$ with $f_{\lambda,\mu_{k}}$ already defined during a previous step, set
\begin{equation*}
v_{\lambda,\mu_{k}}=f_{\lambda,\mu_{k}}\cdot v_{\mu_{k}}.
\end{equation*}

For every other $\lambda$ that satisfies $\mu_k\rightarrow \lambda$ but $\mu_k\not\xrightarrow{2}\lambda$ take
\begin{equation*}
v_{\lambda,\mu_{k}}\in S^\lambda\subset P_{\mu_{k}}
\end{equation*}
to be an arbitrary non-zero vector, and for each $\nu$ that satisfies $\lambda\rightarrow\nu$ but $\mu_k\not\xrightarrow{2}\nu$ define
\begin{equation*}
f_{\nu,\lambda}\in E_\nu'\C\inj E_{\lambda}'
\end{equation*}
to be the unique non-zero vector that satisfies
\begin{equation*}
v_{{\nu,\mu_k}}=f_{\nu,\lambda}\cdot v_{\lambda,\mu_k}.
\end{equation*}
We want to make sure that this does not overwrite any previous definition of $f_{\nu,\lambda}$, i. e. that under the assumption that $f_{\nu,\lambda}$ was defined already during step $m<k$, then $f_{\nu,\mu_k}$ and $f_{\lambda,\mu_k}$ were also defined during an earlier step. Let $\mu_l$ be the join of $\mu_k$ and $\mu_m$ in the Young lattice. It is easily seen that $l<k$, and $\mu_l\rightarrow \nu$ but $\mu_l\not\xrightarrow{2}\nu$. Hence $f_{\nu,\mu_k}$ and $f_{\lambda,\mu_k}$ were defined already during step $l$.

Now let $\mu_{k+1}$ be the smallest diagram greater than $\mu_k$ according to the dominance order (cf. \cite[p. 68]{Sa01}, though any linear order $\trianglelefteq$ satisfying that $|\lambda|<|\mu|$ implies $\lambda\trianglelefteq\mu$ will do). 
\end{description}

We claim that for every diamond $(\lambda_1,\lambda_2,\lambda_3,\lambda_4)$ we have 
\begin{equation*}
f_{\lambda_4,\lambda_2}f_{\lambda_2,\lambda_1}=f_{\lambda_4,\lambda_3}f_{\lambda_3,\lambda_1}.
\end{equation*}
Indeed, let $f_{\lambda_3,\lambda_4}$ have been defined during step $m$. Let $\mu_l$ be the join of $\mu_m$ and $\lambda_2$. Clearly $l\le m$ with $\mu_l\rightarrow\lambda_1$, but $\mu_l\not\xrightarrow{2}\lambda_4$. Therefore
\begin{equation*}
f_{\lambda_4,\lambda_2}f_{\lambda_2,\lambda_1}\cdot v_{\lambda_1,\mu_l}=f_{\lambda_4,\lambda_2}\cdot v_{\lambda_2,\mu_l}=v_{\lambda_4,\mu_l}=f_{\lambda_4,\lambda_3}\cdot v_{\lambda_3,\mu_l}=f_{\lambda_4,\lambda_3}f_{\lambda_3,\lambda_1}\cdot v_{\lambda_1,\mu_l},
\end{equation*}
which implies that
\begin{equation*}
f_{\lambda_4,\lambda_2}f_{\lambda_2,\lambda_1}=f_{\lambda_4,\lambda_1}=f_{\lambda_4,\lambda_3}f_{\lambda_3,\lambda_1},
\end{equation*}
as claimed.

\begin{comment}
\begin{enumerate}
\item[$($i$)$]
For every $n\in\Z_{>0}$ and every diagram $\lambda\vdash n$, fix any non-zero vector $v_\lambda$ in $P_\lambda(\underline{n})$ .

\item[$($ii$)$]
Let $\mu=(1)$ be the diagram consisting of one node. Let $\mathcal{F}=\varnothing$.

\item[$($iii$)$]
For all subquotients $L_{\lambda_3}$ and $L_{\lambda_4}$ of $P_\mu$ with $\lambda_4$ obtained from $\lambda_3$ by adding a node, and where $E'_{\lambda_4}\C\inj e'_{t^1_{\lambda_3}}\cap \mathcal{F}=\varnothing$, pick an arbitrary non-zero vector $f_{\lambda_4,\lambda_3}\in E'_{\lambda_4}\C\inj e'_{t^1_{\lambda_3}}$ which satisfies that for any diamond $(\lambda_1,\lambda_2,\lambda_3,\lambda_4)$ with $f_{\lambda_2,\lambda_1},f_{\lambda_3,\lambda_1},f_{\lambda_4,\lambda_2}\in \mathcal{F}$ we have
\begin{equation*}
f_{\lambda_4,\lambda_2}f_{\lambda_2,\lambda_1}=f_{\lambda_4,\lambda_3}f_{\lambda_3,\lambda_1}.
\end{equation*}
Now redefine
\begin{equation*}
\mathcal{F}:= \mathcal{F}\cup \{f_{\lambda_4,\lambda_3}\},
\end{equation*}
and define $v_{\lambda,\mu}=f_{\lambda,\mu}\cdot v_\mu\in L_\lambda\subset P_\mu$. 

\item[$($iv$)$]
Redefine $\mu$ to be the smallest diagram greater than the previous $\mu$ with respect to the dominance order (any linear order $\trianglelefteq$ satisfying that $|\lambda|<|\mu|$ implies $\lambda\trianglelefteq\mu$ will do). Now return to step (iii). 

\end{enumerate}

\end{comment}

Thus we have that
\begin{align*}
\phi_{\lambda_2,\lambda_1}\circ\phi_{\lambda_4,\lambda_2}:P_{\lambda_4}&\rightarrow P_{\lambda_1} &P_{\lambda_1}&\leftarrow P_{\lambda_4}:\phi_{\lambda_3,\lambda_1}\circ\phi_{\lambda_4,\lambda_3}\\
v_{\lambda_4}&\mapsto f_{\lambda_4,\lambda_2}f_{\lambda_2,\lambda_1}v_{\lambda_1}&=f_{\lambda_4,\lambda_3}f_{\lambda_3,\lambda_1}v_{\lambda_1}&\mapsfrom v_{\lambda_4}.
\end{align*}

\end{proof}

\section{On the Koszulity of the category of injections between finite sets}\label{s3}
All modules in this section will be considered as $\Z$-graded ones, with the grading obtained as described in sections \ref{s1} and \ref{ss23}. 

\subsection{Some additional notation}

Let us in this section denote by 
\begin{equation*}
P_\lambda=\C\cat(\lambda,\_) 
\end{equation*}
the indecomposable projective object in $\cgmod$ that is generated in degree 0 and starting at diagram $\lambda$, and let $L_\lambda$ be the (simple, one-dimensional) top of $P_\lambda$. The same notation was already used for $\C\inj$-modules in the previous section, but observe that if we forget about the grading, these different modules correspond to each other under the Morita equivalence given in Theorem \ref{quiverthm}.

As before, we let the (possibly decorated) letters $\lambda$, $\mu$ and $\nu$ denote Young diagrams, but when no confusion should occur we will by abuse of notation also let $\lambda$ denote a subquotient, isomorphic to $L_\lambda\langle i\rangle$, of some module and for some $i$. 

\subsection{Koszul categories and quadratic duals}
\label{koszss}
Following \cite{MOS09}, we define Koszul categories and their quadratic duals in analogy with the definitions for finite-dimensional, unital algebras.

A $\Z_{\ge 0}$-graded category $\mathcal{E}$ is said to be \emph{Koszul} if every simple object $L$ in $\mathcal{E}\text{-gMod}$ satisfying that $L_0=L$ has a \emph{linear resolution}, i.e. a graded projective resolution $\mathcal{P}^\bullet$ such that each of its modules $\mathcal{P}^{-n}$ is generated by $(\mathcal{P}^{-n})_n$. 

We note that the Young lattice is locally Cohen-Macaulay by property 5 in \cite{BS05}, so from the main result of \cite{Po95} the following proposition immediately follows. 

\begin{myprop}
$\C\cat'$ is Koszul. 
\end{myprop}

A positively graded (in the sense of \cite[Definition 1]{MOS09}) $\C$-linear category $\C\mathcal{E}$ has a \emph{quadratic dual} defined as follows. Let $\C\mathcal{E}_i$ be the dense subcategory of $\C\mathcal{E}$ containing precisely the morphisms that are homogeneous in degree $i$. Then $\C\mathcal{E}_i(\_,\_)$ is a $\C\mathcal{E}_0$-bimodule in the natural way, and so is therefore $(\C\mathcal{E}_i(\_,\_))^*$. Let $\C\mathcal{F}$ be the category with the same objects as $\C\mathcal{E}$ and morphism spaces
\begin{equation*}
\C\mathcal{F}(\_,\_)=\C\mathcal{E}_0(\_,\_)\oplus (\C\mathcal{E}_1(\_,\_))^*\oplus (\C\mathcal{E}_1(\_,\_))^*\otimes_{\C\mathcal{E}_0} (\C\mathcal{E}_1(\_,\_))^*\oplus\dots
\end{equation*}
For $X,Y\in\Ob(\C\mathcal{E})$, the multiplication map
\begin{equation*}
m_{X,Y}:\bigoplus_{Z\in\Ob (\C\mathcal{E})}\C\mathcal{E}_1(Z,Y)\otimes_{\C\mathcal{E}_0}\C\mathcal{E}_1(X,Z)\rightarrow \C\mathcal{E}_2(X,Y)
\end{equation*}
gives rise to the dual map
\begin{align*}
m_{X,Y}^*:&(\C\mathcal{E}_2(X,Y))^*\rightarrow (\bigoplus_{Z\in\Ob (\C\mathcal{E})}\C\mathcal{E}_1(Z,Y)\otimes_{\C\mathcal{E}_0}\C\mathcal{E}_1(X,Z))^*\\
&\cong \bigoplus_{Z\in\Ob (\C\mathcal{E})}(\C\mathcal{E}_1(X,Z))^*\otimes_{\C\mathcal{E}_0}(\C\mathcal{E}_1(Z,Y))^*.
\end{align*}
Let $\mathcal{J}$ be the category ideal of $\C\mathcal{F}$ generated by the $\im(m_{X,Y}^*)$, and finally define the \emph{quadratic dual} of $\C\mathcal{E}$ to be the category
\begin{equation*}
\C\mathcal{E}^!= \C\mathcal{F}/\mathcal{J}.
\end{equation*}

\subsection{Signs of quiver arrows}
\begin{mylem}
\label{sgnlem}
Each arrow, say from $\lambda$ to $\mu$ of the Young quiver $Q$ may be assigned a ``sign'' $s^{\lambda}_{\mu}=\pm 1$ such that for any diamond $(\lambda_1,\lambda_2,\lambda_3,\lambda_4)$ we have $s^{\lambda_2}_{\lambda_4}s^{\lambda_1}_{\lambda_2}=-s^{\lambda_3}_{\lambda_4}s^{\lambda_1}_{\lambda_3}$.
\end{mylem}
\begin{figure}
\includegraphics{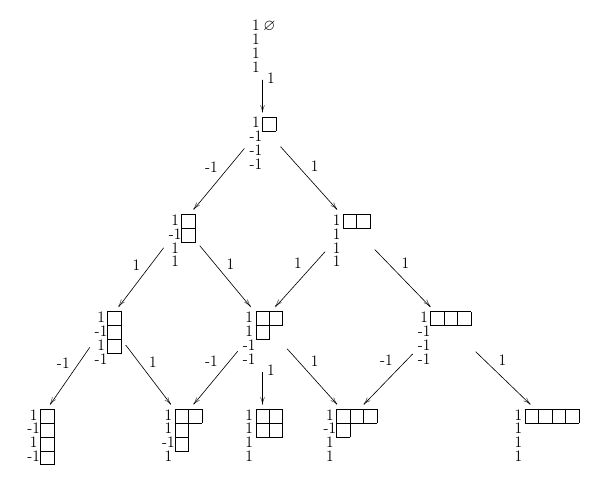}
\begin{comment}
\xymatrix{
&&&{\begin{ytableau}\none[$1$]&\none[\varnothing]\\\none[$1$]\\\none[$1$]\\\none[$1$]\end{ytableau}}\ar[d]^{$1$}&&&\\
&&&{\begin{ytableau}\none[$1$]&{}\\\none[$-1$]\\\none[$-1$]\\\none[$-1$]\end{ytableau}}\ar[dl]_{$-1$}\ar[dr]^{$1$}&&&\\
&&{\begin{ytableau}\none[$1$]&{}\\\none[$-1$]&{}\\\none[1]\\\none[$1$]\end{ytableau}}\ar[dl]_{$1$}\ar[dr]^{$1$}&&{\begin{ytableau}\none[$1$]&{}&{}\\\none[$1$]\\\none[$1$]\\\none[$1$]\end{ytableau}}\ar[dl]_{$1$}\ar[dr]^{$1$}&&\\
&{\begin{ytableau}\none[$1$]&{}\\\none[$-1$]&{}\\\none[1]&{}\\\none[$-1$]\end{ytableau}}\ar[dl]_{$-1$}\ar[dr]^{$1$}&&{\begin{ytableau}\none[$1$]&{}&{}\\\none[$1$]&{}\\\none[$-1$]\\\none[$-1$]\end{ytableau}}\ar[d]^{$1$}\ar[dl]_{$-1$}\ar[dr]^{$1$}&&{\begin{ytableau}\none[$1$]&{}&{}&{}\\\none[$-1$]\\\none[$-1$]\\\none[$-1$]\end{ytableau}}\ar[dl]_{$-1$}\ar[dr]^{$1$}&\\
{\begin{ytableau}\none[$1$]&{}\\\none[$-1$]&{}\\\none[$1$]&{}\\\none[$-1$]&{}\end{ytableau}}&&{\begin{ytableau}\none[$1$]&{}&{}\\\none[$1$]&{}\\\none[$-1$]&{}\\\none[$1$]\end{ytableau}}&{\begin{ytableau}\none[$1$]&{}&{}\\\none[$1$]&{}&{}\\\none[$1$]\\\none[$1$]\end{ytableau}}&{\begin{ytableau}\none[$1$]&{}&{}&{}\\\none[$-1$]&{}\\\none[$1$]\\\none[$1$]\end{ytableau}}&&{\begin{ytableau}\none[$1$]&{}&{}&{}&{}\\\none[$1$]\\\none[$1$]\\\none[$1$]\end{ytableau}}
}
\end{comment}
\caption{The signs of quiver arrows (and diagram rows) assigned by the algorithm of Lemma \ref{sgnlem} up to diagrams size four.}
\end{figure}
\begin{proof}
Consider each diagram as embedded in a doubly infinite sequence of rows. Assign to the row of the empty diagram the ``sign'' $1$. Assign signs $1$ or $-1$ to the rows of the rest of the diagrams as follows. Going along an arrow of $Q$ corresponds to adding an addable node to a diagram $\lambda$, thus obtaining another diagram $\mu$. When the rows of $\lambda$ have been given signs, let the rows of the new $\mu$ get the same signs, but with signs switched of the rows below the one where the node was added. Also let $s^{\lambda}_{\mu}$ be the sign of that row. 

It is easily seen that this procedure assigns a well-defined sign to every arrow. For any diamond $(\lambda_1,\lambda_2,\lambda_3,\lambda_4)$, precisely one of the pairs $s^{\lambda_2}_{\lambda_4}$ and $s^{\lambda_1}_{\lambda_2}$, respective $s^{\lambda_3}_{\lambda_4}$ and $s^{\lambda_1}_{\lambda_3}$ will be of the same sign (namely the pair corresponding to adding the uppermost node last). The lemma follows. 
\end{proof}

The above lemma will be used in the proof of Theorem \ref{koszthm}, but also gives us the following result.

\begin{mythm}
\label{selfdualthm}
There is an isomorphism $\C\cat^!\cong(\C\cat)^{\text{op}}$. 
\end{mythm}
\begin{proof}
If we set $\C\mathcal{E}=\C\cat$ in the definition of the quadratic dual from Subsection \ref{koszss}, then the category $\C\mathcal{F}$ is readily seen to be isomorphic to $(\C\cat')^{\text{op}}$. Pick a basis for $\C\cat_1$ consisting of morphisms of the form $f_{\mu,\nu}\in\C\cat_1(\mu,\nu)$ such that for each diamond $(\mu,\nu_1,\nu_2,\lambda)$ we have $f_{\nu_1,\lambda}f_{\mu,\nu_1}=f_{\nu_2,\lambda}f_{\mu,\nu_2}$. The relations $\im(m_{\mu,\lambda}^*)=0$ imposed on $\C\mathcal{F}$ in order to obtain $\C\cat^!$ are of three different kinds depending on $\mu$ and $\lambda$. 

If $\mu\xrightarrow{2}\lambda$, then $m_{\mu,\lambda}=0$, so $m_{\mu,\lambda}^*=0$ and $\dim(\C\cat^!(\lambda,\mu))=\dim(\C\cat'(\mu,\lambda))=1$. 

If $\mu^T\xrightarrow{2}\lambda^T$, then $m_{\mu,\lambda}$ is an isomorphism, and $m_{\mu,\lambda}^*$ is surjective, hence $\C\cat^!(\lambda,\mu)=0$. 

Finally, assume that neither of the above cases holds. Then there is a diamond $(\mu,\nu_1,\nu_2,\lambda)$ such that 
\begin{align*}
\dom(m_{\mu,\lambda})&=\bigoplus_{\nu\in\Ob (\C\cat)}\C\cat_1(\nu,\lambda)\otimes_{\C\cat_0}\C\cat_1(\mu,\nu)\\&= (\C\cat_1(\nu_1,\lambda)\otimes_{\C\cat_0}\C\cat_1(\mu,\nu_1))\oplus (\C\cat_1(\nu_2,\lambda)\otimes_{\C\cat_0}\C\cat_1(\mu,\nu_2)).
\end{align*}
Because we have commutativity of morphisms between the same objects in $\C\cat$, we may pick bases $\{f_{\nu_1,\lambda}\otimes_{\C\cat_0}f_{\mu,\nu_1},f_{\nu_2,\lambda}\otimes_{\C\cat_0}f_{\mu,\nu_2}\}$ in $\dom(m_{\mu,\lambda})$ and $\{f_{\nu_1,\lambda}f_{\mu,\nu_1}\}$ in $\C\cat_2(\mu,\lambda)$ for which 
\begin{equation*}
m_{\mu,\lambda}=\begin{pmatrix} 1&1\end{pmatrix}
\end{equation*}
and thus in the dual bases $\{f_{\mu,\nu_1}^*\otimes_{\C\cat_0}f_{\nu_1,\lambda}^*,f_{\mu,\nu_2}^*\otimes_{\C\cat_0}f_{\nu_2,\lambda}^*\}$ and $\{(f_{\nu_1,\lambda}f_{\mu,\nu_1})^*\}$
\begin{equation*}
m_{\mu,\lambda}^*=\begin{pmatrix} 1&1\end{pmatrix}^T=\begin{pmatrix} 1\\ 1\end{pmatrix}.
\end{equation*}
Therefore we obtain the anticommutativity relation
\begin{equation*}
f_{\mu,\nu_1}^*\otimes_{\C\cat_0}f_{\nu_1,\lambda}^*+f_{\mu,\nu_2}^*\otimes_{\C\cat_0}f_{\nu_2,\lambda}^*=0
\end{equation*}
in $\C\cat^!$. 

Then clearly the functor 
\begin{align*}
F:\C\cat&\rightarrow \C\cat^!\\
\lambda&\mapsto\lambda^T\\
f_{\mu,\lambda}&\mapsto s^\mu_\lambda f_{\mu^T,\lambda^T}^*,
\end{align*}
where the signs $s^\mu_\lambda$ are picked as in Lemma \ref{sgnlem}, is a duality of categories that is bijective on objects. The desired result follows. 
\end{proof}
We note that a very similar proof gives us also the following description of the quadratic dual of the (linearized) Young lattice. 
\begin{myprop} \label{ppp1}
There is an isomorphism $\C{\cat'}^!\cong (\C\cat'/\mathcal{J})^{\text{op}}$, where $\mathcal{J}$ is the ideal generated by all $\C\cat'(\mu,\lambda)$ with $\mu\xrightarrow{2}\lambda$ or $\mu^T\xrightarrow{2}\lambda^T$.
\end{myprop}

\subsection{Linear resolutions for simple $\C\cat$-modules}\label{sss2}
We will in this section construct explicit linear resolutions of the simple $\C\cat$-modules. 

Fix some diagram $\xi$. Let $I_i$ be the set of diagrams that can be obtained by adding $-i$ nodes, no two of which to the same row, to $\xi$. Define
\begin{equation*}
\mathcal{P}^i=
\begin{cases}
        \bigoplus_{\lambda\in I_i}P_{\lambda}\langle i\rangle, & \mbox{for } i\le 0\\
        0, & \mbox{ for $i>0$.}
        \end{cases} 
\end{equation*} 

Fix non-zero elements $v^{\lambda,i}_{\lambda}\in P_{\lambda}\langle i\rangle(\lambda)$. Consider a diagram $\mu$. If $\mu$ is a subquotient of $P_\lambda\langle i\rangle$, we have a uniquely determined $x\in\cat(\lambda,\mu)$. Otherwise set $x=0$.  Finally define $v^{\lambda,i}_{\mu}=x\cdot v^{\lambda,i}_{\lambda}$, which is a basis element of the subquotient $\mu$  in $P_{\lambda}\langle i\rangle$, provided that $\mu$ is a subquotient of $P_{\lambda}$, and 0 otherwise. 

Define the maps 
\begin{align*}
\pi_{i,\lambda}: \mathcal{P}^i&\rightarrow \mathcal{P}^{i+1}\\
v^{\mu,i}_{\mu}&\mapsto s^{\lambda}_{\mu}v^{\lambda,i+1}_{\mu},
\end{align*}
where the sign $s^\lambda_\mu$ is the one defined in Lemma \ref{sgnlem}, and also
\begin{equation*}
\delta^i=\begin{cases}
        \bigoplus_{\lambda\in I_{i+1}}\pi_{i,\lambda}, & \mbox{for } i< 0\\
        0, & \mbox{ for $i\ge 0$.}
        \end{cases}
\end{equation*}

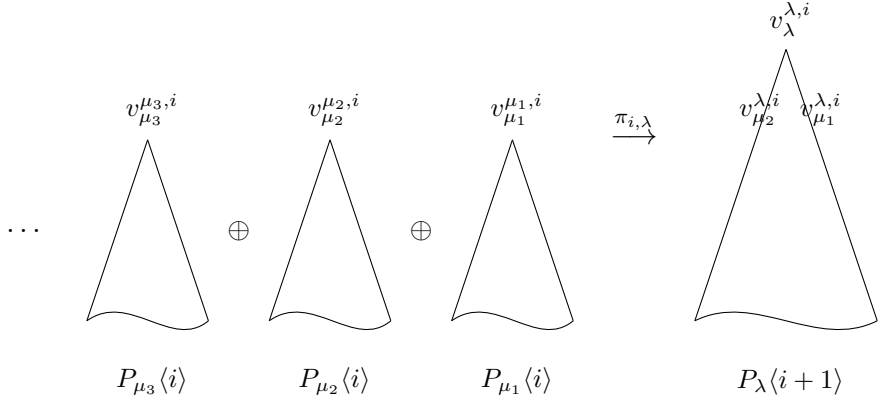
\begin{figure}
\special{em:linewidth 0.4pt} \unitlength 0.80mm
%(a)\quad\quad
\begin{picture}(140.00,75.00)
%%%%%%%%%%%%%%%%%%%%%%%%%%%%%%%%%%%%%%%%%%%%%%%%%%%%%%%%
%\put(5,50){\line(1,-2){25}}
%\Line(35,0)(50,30)
%\cbezier[-10](5,0)(10,10)(20,5)(30,0)
\Line(10,15)(20,45)
\Line(20,45)(30,15)
\cbezier[1](10,15)(17,20)(24,10)(30,15)
\put(20,5){\makebox(0,0)[cc]{ $P_{\mu_3}\langle i\rangle$}}
\put(20,50){\makebox(0,0)[cc]{ $v_{\mu_3}^{\mu_3,i}$}}

\put(100,47){\makebox(0,0)[cc]{\large $\overset{\pi_{i,\lambda}}\longrightarrow$}}
\put(35,30){\makebox(0,0)[cc]{\large $\oplus$}}
\put(65,30){\makebox(0,0)[cc]{\large $\oplus$}}
\put(0,30){\makebox(0,0)[cc]{\large $\dots$}}

\Line(70,15)(80,45)
\Line(80,45)(90,15)
\cbezier[1](70,15)(77,20)(84,10)(90,15)
\put(80,5){\makebox(0,0)[cc]{ $P_{\mu_1}\langle i\rangle$}}
\put(80,50){\makebox(0,0)[cc]{ $v_{\mu_1}^{\mu_1,i}$}}

\Line(40,15)(50,45)
\Line(50,45)(60,15)
\cbezier[1](40,15)(47,20)(54,10)(60,15)

%\put(20,5){\makebox(0,0)[cc]{ $P_{\mu_1}\langle i\rangle$}}
\put(50,5){\makebox(0,0)[cc]{ $P_{\mu_2}\langle i\rangle$}}
\put(50,50){\makebox(0,0)[cc]{ $v_{\mu_2}^{\mu_2,i}$}}

\Line(110,15)(125,60)
\Line(125,60)(140,15)
\cbezier[1](110,15)(120,20)(130,10)(140,15)
\put(125,5){\makebox(0,0)[cc]{ $P_{\lambda}\langle i+1\rangle$}}
\put(125,65){\makebox(0,0)[cc]{ $v_{\lambda}^{\lambda,i}$}}

\put(120,50){\makebox(0,0)[cc]{ $v_{\mu_2}^{\lambda,i}$}}

\put(130,50){\makebox(0,0)[cc]{ $v_{\mu_1}^{\lambda,i}$}}

%{\color{gray}\Line(120,5)(135,50)
%\Line(135,50)(150,5)
%\cbezier[1](120,5)(130,10)(140,0)(150,5)
%}

%%%%%%%%%%%%%%%%%%%%%%%%%%%%%%%%%%%%%%%%%%%%%%%%%%%%%%%%
\end{picture}
\caption{The map $\pi_{i,\lambda}$ maps $\mathcal{P}^i$ to the component $P_\lambda\langle i+1\rangle$ only. In the picture, only $\mu_1$ and $\mu_2$ occur as subquotients of $P_{\lambda}\langle i+1\rangle$ at the degree $i$ level.}
\end{figure}
\begin{mythm}
\label{koszthm}
The modules $\mathcal{P}^i$ and maps $\delta^i$ form a linear resolution, $\mathcal{P}^\bullet$, of $L_{\xi}$. 
\end{mythm}
\begin{proof}
We divide the proof into several parts, the first two of which are obvious.

\begin{enumerate}[leftmargin=*]
\item[$($i$)$]

\emph{If $\mathcal{P}^\bullet$ is a resolution, then it is a linear one. }

\item[$($ii$)$]

\emph{The cohomology at position 0 is $L_{\xi}$.}

\end{enumerate}

It remains to consider the cohomology positions less than 0. Fix therefore for the remainder of the proof some $i\le -2$. 

\begin{enumerate}[leftmargin=*] 
\item[$($iii$)$]

\emph{$\mathcal{P}^\bullet$ is a complex.}\\ 
It suffices to check that $\pi_{i+1,\lambda}\circ\delta^i(v^{\nu,i}_{\nu})=0$ for all $\nu\in I_i$ and $\lambda\in I_{i+2}$. It will never be the case that $\nu$ is obtained from $\lambda$ by adding 2 nodes to the same row, by construction of $\mathcal{P}^i$. 

If $\nu$ can not be obtained from $\lambda$ at all by adding 2 nodes, or if $\nu$ is obtained from $\lambda$ by adding 2 nodes to the same column, then the result is immediate because there is no subquotient $\nu$ in $P_{\lambda}$. In the remaining case, there exist unique $\mu_1,\mu_2\in I_{i+1}$ that are subquotients of $P_{\lambda}$ such that $(\lambda,\mu_1,\mu_2,\nu)$ is a diamond. Then 
\begin{equation*}
\pi_{i+1,\lambda}\circ\delta^i(v^{\nu,i}_{\nu})=s^{\lambda}_{\mu_1}s^{\mu_1}_{\nu}v^{\lambda,i+2}_{\nu}+s^{\lambda}_{\mu_2}s^{\mu_2}_{\nu}v^{\lambda,i+2}_{\nu}=0,
\end{equation*}
using Lemma \ref{sgnlem}. 

\end{enumerate}

We want to show that we conversely have $\ker(\delta^{i+1})\subset\im(\delta^i)$, so that $\mathcal{P}^\bullet$ is exact at positions less than 0. We first show only a partial statement.

\begin{enumerate}[leftmargin=*] 
\item[$($iv$)$]

\emph{We have $\ker(\delta^{i+1})\cap(\mathcal{P}^{i+1})_{-i}\subset\im(\delta^i)$.}\\
 Note that $\ker(\delta^{i+1})\cap(\mathcal{P}^{i+1})_{-i}$ is spanned by vectors of the form 
\begin{equation*} 
v_\nu=\sum_{\mu\in I_{i+1}}k_{\mu}v^{\mu,i+1}_{\nu},
\end{equation*}
where $\nu$ is some diagram obtained by adding a node to some $\mu\in I_{i+1}$. Assuming that $v_\nu\ne 0$, we get two cases depending on whether $\nu\in I_i$ or not. 

If $\nu\not\in I_i$, then there must be precisely one $\mu\in I_{i+1}$ such that $v_{\nu}^{\mu,i+1}\ne 0$: By assumption there is at least one such $\mu$, which must be obtained from $\nu$ by removing a node from a row where $\nu$ has two more nodes than $\xi$. But if there were to exist more than one such row, then $\mu\not\in I_{i+1}$. Therefore the row and hence $\mu$ is uniquely determined. 
That 
\begin{equation*}
\delta^{i+1}(v_\nu)=0
\end{equation*}
then implies that for all $\lambda\in I_{i+2}$ with $\lambda\rightarrow\mu$ also
\begin{equation*}
0=\pi_{i+1,\lambda}(k_\mu v^{\mu,i+1}_{\nu})=s^\lambda_\mu k_\mu v^{\lambda,i+1}_{\nu},
\end{equation*}
so that $k_{\mu}=0$ and hence $v_\nu=0\in\im(\delta^i)$. 

If on the other hand $\nu\in I_i$, again fix some $\mu\in I_{i+1}$ with $v^{\mu,i+1}_{\nu}\ne 0$. For any other $\mu'\in I_{i+1}$ with $\mu'\rightarrow\nu$, there is a diamond $(\lambda,\mu,\mu',\nu)$. Then 
\begin{equation*}
\delta^{i+1}(v_\nu)=0
\end{equation*}
implies 
\begin{equation*}
0=\pi_{i+1,\lambda}(v_\nu)=k_{\mu}s^{\lambda}_{\mu}s^{\mu}_{\nu}v^{\lambda,i+2}_{\nu}+k_{\mu'}s^{\lambda}_{\mu'}s^{\mu'}_{\nu}v^{\lambda,i+2}_{\nu},
\end{equation*}
so if $v_\nu^{\mu',i+1}\ne 0$, we get
\begin{equation*}
k_{\mu'}=-\frac{k_\mu s^{\lambda}_{\mu}s^{\mu}_{\nu}}{s^{\lambda}_{\mu'}s^{\mu'}_{\nu}}.
\end{equation*}
In particular, $v_\nu$ is unique up to a scalar, hence must be a scalar multiple of $\delta^{i}(v^{\nu,i}_{\nu})$. 

\end{enumerate}

The previous part and the next part together imply that $\ker(\delta^{i+1})\subset\im(\delta^i)$.

\begin{enumerate}[leftmargin=*] 
\item[$($v$)$]

%$-i$ nodes to $\xi$. In fact, we may assume that $\nu\in I_i$, since otherwise there must for any $\mu\in I_{i+1}$ for which $v^{\mu,i+1}_{\nu}\ne 0$ exist $\lambda\in I_{i+2}$ such that $\mu$ is obtained by adding a node to $\lambda$ and $\nu$ is obtained from $\mu$ by adding a node to the same row, in which case $0=\pi_{i+1,\lambda}(v)=\pm c_{\mu}v^{\lambda,i+2}_{\nu}$, forcing $c_{\mu}=0$ anyway. Then for any $\mu_1,\mu_2\in I_{i+1}$ with $v^{\mu_1,i+1}_{\nu},v^{\mu_2,i+1}_{\nu}\ne 0$ there is a diamond $(\lambda,\mu_1,\mu_2,\nu)$, so that $0=\pi_{i+1,\lambda}(v)=c_{\mu_1}s^{\lambda}_{\mu_1}s^{\mu_1}_{\nu}v^{\lambda,i+2}_{\nu}+c_{\mu_2}s^{\lambda}_{\mu_2}s^{\mu_2}_{\nu}v^{\lambda,i+2}_{\nu}$. This implies that one of the coefficients $c_{\mu}$ determines the rest, so $v$ is for fixed $\nu$ unique up to a scalar, hence must be a scalar multiple of $\delta^{i}(v^{\nu,i}_{\nu})$. 

\emph{$\ker(\delta^{i+1})$ is generated in degree $-i$.}\\
For an arbitrary $\lambda\in I_{i+2}$, let $I_{i+1}^{\lambda}\subset I_{i+1}$ be the subset consisting of diagrams obtainable from $\lambda$ by adding one node. Then there is a commutative diagram
\begin{equation*}
\xymatrixcolsep{5pc}
\xymatrix{
\mathcal{P}^{i+1}\ar[r]^-{\pi_{i+1,\lambda}}\ar[d]_{\text{projection}}&\mathcal{P}^{i+2}\\
\bigoplus_{\mu\in I_{i+1}^{\lambda}}P_{\mu}\langle i+1\rangle \ar[ur]_-{\varphi'_{i+1,\lambda}}\ar@{->>}[r]_-{\varphi_{i+1,\lambda}}& (\bigoplus_{\mu\in I_{i+1}^{\lambda}}P_{\mu}\langle i+1\rangle)/(\ker(\varphi'_{i+1,\lambda})_{-i}\ar@{^{(}->}[u]
}
\end{equation*}

\begin{figure}
\special{em:linewidth 0.4pt} \unitlength 0.80mm
%(a)\quad\quad
\begin{picture}(180.00,75.00)
%%%%%%%%%%%%%%%%%%%%%%%%%%%%%%%%%%%%%%%%%%%%%%%%%%%%%%%%
%\put(5,50){\line(1,-2){25}}
%\Line(35,0)(50,30)
%\cbezier[-10](5,0)(10,10)(20,5)(30,0)

\put(100,47){\makebox(0,0)[cc]{\large $\overset{\varphi'_{i,\lambda}}\longrightarrow$}}

\put(65,30){\makebox(0,0)[cc]{\large $\oplus$}}

\Line(70,15)(80,45)
\Line(80,45)(90,15)
\cbezier[1](70,15)(77,20)(84,10)(90,15)
\put(80,5){\makebox(0,0)[cc]{ $P_{\mu_1}\langle i\rangle$}}
\put(80,50){\makebox(0,0)[cc]{ $v_{\mu_1}^{\mu_1,i}$}}

\Line(40,15)(50,45)
\Line(50,45)(60,15)
\cbezier[1](40,15)(47,20)(54,10)(60,15)

%\put(20,5){\makebox(0,0)[cc]{ $P_{\mu_1}\langle i\rangle$}}
\put(50,5){\makebox(0,0)[cc]{ $P_{\mu_2}\langle i\rangle$}}
\put(50,50){\makebox(0,0)[cc]{ $v_{\mu_2}^{\mu_2,i}$}}

\Line(110,15)(125,60)
\Line(125,60)(140,15)
\cbezier[1](110,15)(120,20)(130,10)(140,15)
\put(125,5){\makebox(0,0)[cc]{ $P_{\lambda}\langle i+1\rangle$}}
\put(125,65){\makebox(0,0)[cc]{ $v_{\lambda}^{\lambda,i+1}$}}

\put(120,50){\makebox(0,0)[cc]{ $v_{\mu_2}^{\lambda,i+1}$}}

\put(132,50){\makebox(0,0)[cc]{ $v_{\mu_1}^{\lambda,i+1}$}}

%{\color{gray}\Line(120,5)(135,50)
%\Line(135,50)(150,5)
%\cbezier[1](120,5)(130,10)(140,0)(150,5)
%}

%%%%%%%%%%%%%%%%%%%%%%%%%%%%%%%%%%%%%%%%%%%%%%%%%%%%%%%%
\end{picture}
\caption{The map $\varphi'_{i,\lambda}$ maps $\bigoplus_{\mu\in I^\lambda_i} P_\mu\langle i\rangle$ to $P_\lambda\langle i+1\rangle$. In the picture we have $I_i^\lambda=\{ \mu_1,\mu_2\}$.}
\end{figure}
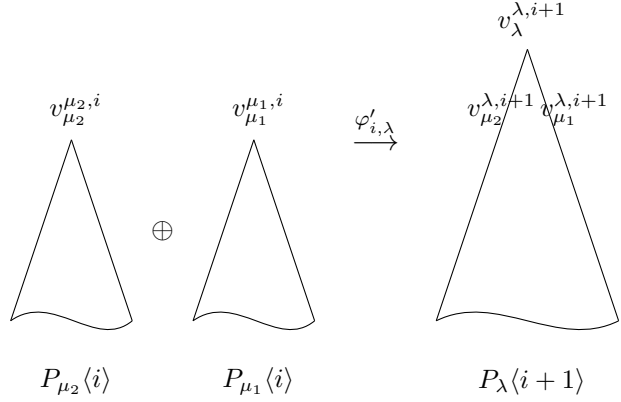
The desired result follows if we can show that $\im(\varphi'_{i+1,\lambda})=\im(\varphi_{i+1,\lambda})$. 

Clearly $\varphi'_{i+1,\lambda}$ is injective in degree $-i-1$, so in particular we have that as $\C\cat$-modules 
\begin{equation*}
\im(\varphi'_{i+1,\lambda})=\langle v_{\mu}^{\lambda,i+2}|\mu\in I_{i+1}\rangle\subset P_{\lambda}\langle i+2\rangle.
\end{equation*}
Thus $\im(\varphi'_{i+1,\lambda})$ will contain as subquotients precisely one copy of each diagram $\zeta$ satisfying $\mu\rightarrow\zeta$ but $\lambda\not\xrightarrow{2}\zeta$ for some $\mu\in I_{i+1}^{\lambda}$. We will show that $\im \varphi_{i+1,\lambda}$ contains no more subquotients than this. The argument is further subdivided into three parts. 

\begin{enumerate}
\item

We will below have use of the following argument. Let $\mu_1,\mu_2\in I_{i+1}^{\lambda}$ be different with $\nu$ a diagram satisfying $\mu_1\rightarrow\nu$ and $\mu_2\rightarrow\nu$. Let further $x\in\cat(\mu_1,\nu)$ and $y\in \cat(\mu_2,\nu)$. There is a diamond $(\lambda,\mu_1,\mu_2,\lambda')$, and we have in $\cat$ the commutative diagram
\begin{equation*}
\xymatrix{
\mu_1\ar[ddr]_-x\ar[dr]&&\mu_2\ar[ddl]^-y\ar[dl]\\
&\lambda'\ar[d]^-z&\\
&\nu&
}
\end{equation*}

Using that $P_\lambda\langle i+2\rangle$ has only a single copy of $\lambda'$ as a subquotient, and that $v_{\lambda'}^{\mu_1,i+1}, v_{\lambda'}^{\mu_2,i+1}\in (\mathcal{P}^{i+1})_{-i}$, we obtain
\begin{align*}
\varphi_{i+1,\lambda}(\C v_{\nu}^{\mu_1,i+1})&=z\varphi_{i+1,\lambda}(\C v_{\lambda'}^{\mu_1,i+1})\\
&=z\varphi'_{i+1,\lambda}(\C v_{\lambda'}^{\mu_1,i+1})\\
&=z\C v_{\lambda'}^{\lambda,i+2}\\
&=z\varphi'_{i+1,\lambda}(\C v_{\lambda'}^{\mu_2,i+1})\\
&=z\varphi_{i+1,\lambda}(\C v_{\lambda'}^{\mu_2,i+1})\\
&=\varphi_{i+1,\lambda}(\C v_{\nu}^{\mu_2,i+1}).
\end{align*}

\item

In particular we get from part (a) that if $\nu$ is a subquotient of both $P_{\mu_1}\langle i+1\rangle$ and $P_{\mu_2}\langle i+1\rangle$, then these are identified by $\varphi_{i+1,\lambda}$, so that $\im(\varphi_{i+1,\lambda})$ contains at most one copy of $\nu$ as a subquotient. 

\item

Let $\lambda\xrightarrow{2}\nu$, with $\nu$ having at least two more nodes in column $k$ than $\lambda$. \emph{Then $\im(\varphi_{i+1,\lambda})$ has no subquotient $\nu$.} \\
In order to show this, we show that whenever there is a $\mu_1\in I_{i+1}^{\lambda}$ such that $\nu$ is a subquotient of $P_{\mu_1}\langle i+1\rangle$, we have $\varphi_{i+1,\lambda}(\C v_{\nu}^{\mu_1,i+1})=0$. For $\nu$ to be a subquotient of $P_{\mu_1}\langle i+1\rangle$ it is necessary that $\mu_1$ is obtained from $\lambda$ by adding a node to column $k$. One of two cases will hold. 

\begin{figure}
\special{em:linewidth 0.4pt} \unitlength 0.80mm
%(a)\quad\quad
\begin{picture}(180.00,75.00)

\Line(80,15)(130,15)
\Line(110,15)(110,25)
\Line(110,25)(120,25)
\Line(120,25)(120,35)
\Line(120,35)(130,35)
\Line(130,15)(130,55)
\Line(120,15)(120,25)
\Line(120,25)(130,25)
\put(125,30){\makebox(0,0)[cc]{ $b_1$}}
\put(125,20){\makebox(0,0)[cc]{ $b_3$}}
\put(115,20){\makebox(0,0)[cc]{ $b_2$}}
\put(125,5){\makebox(0,0)[cc]{ $k$}}
\put(100,35){\makebox(0,0)[cc]{ $\lambda$}}
\put(125,10){\makebox(0,0)[cc]{ $\uparrow$}}

\end{picture}
\caption{Illustration of the diagrams of step (c). If we to column $k$ add node $b_1$ to $\lambda$ we obtain $\mu_1$. In the first case considered, $b_2$ belongs already to $\lambda$, and the diagram $\lambda'$ is obtained by adding node $b_3$. In the second case, node $b_2$ does not belong to $\lambda$, but is perhaps the node added to $\lambda$ in order to obtain $\mu_2$. }
\end{figure}
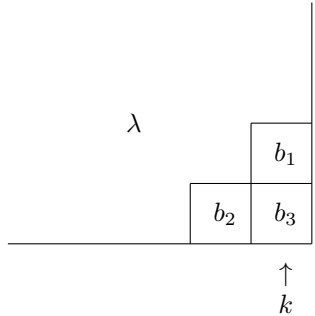

The first is that it is possible to add to $\mu_1$ a node to column $k$ and thereby obtain another diagram $\lambda'$. Let $a\in\cat(\mu_1,\lambda')$. Then, using that $v_{\lambda'}^{\mu_1,i+1}\in (\mathcal{P}^{i+1})_{-i}$, we get
\begin{align*}
\varphi_{i+1,\lambda}(\C v_{\nu}^{\mu_1,i+1})&=a\varphi_{i+1,\lambda}(\C v_{\lambda'}^{\mu_1,i+1})=\\a\varphi'_{i+1,\lambda}(\C v_{\lambda'}^{\mu_1,i+1})&=a\C v_{\lambda'}^{\lambda,i+1}=0.
\end{align*}

The second case is that one has to add to $\mu_1$ at least one node to a column to the left of $k$ before more nodes can be added to $k$. Let $\mu_2$ be obtained by adding one such node to $\lambda$. Thus $\mu_1\rightarrow\nu$ and $\mu_2\rightarrow\nu$. Note that $\nu$ is not a subquotient of $P_{\mu_2}\langle i+1\rangle$. Hence we may again apply part (a) to obtain 
\begin{equation*}
\varphi_{i+1,\lambda}(\C v_{\nu}^{\mu_1,i+1})=\varphi_{i+1,\lambda}(\C v_{\nu}^{\mu_2,i+1})=0,
\end{equation*}
and we are done. 
\end{enumerate}

\end{enumerate}

%Finally we show that in fact $\ker(\delta^{i+1})$ is generated in degree $-i$, so that the inclusion $\ker(\delta^{i+1})\cap(\mathcal{P}^{i+1})_{-i}\subset\im(\delta^i)$ proved in the previous paragraph implies $\ker(\delta^{i+1})\subset\im(\delta^i)$. Let $K\subset\mathcal{P}^{i+1}$ be the submodule that is generated by the degree $-i$ elements of $\ker(\delta^{i+1})$. Then $\delta^{i+1}$ induces a surjection $\mathcal{P}^{i+1}/K\twoheadrightarrow\mathcal{P}^{i+2}$, which clearly restricts to an injection in degree $-(i+1)$. It now follows from the universal property of projective objects that in fact $\mathcal{P}^{i+1}/K\cong\mathcal{P}^{i+2}$, which is what we needed to show. This completes the proof. 

%The above in particular shows that $\dim(\ker(\delta^{i+1})\cap\delta^i(\spann\{v^{\nu,i}_{\nu}\}))=1$ for any $\nu\in I_i$. Note that $\ker(\delta^{i+1})=\bigoplus_{\nu\in I_i}\ker(\delta^{i+1})\cap\spann\{v^{\mu,i+1}_{\nu}|\mu\in I_{i+1}\}\bigoplus \ker(\delta^{i+1})\cap\spann\{v^{\mu,i+1}_{\nu'}|\mu\in I_{i+1}\}$, where $\nu'$ We next show that in fact $\dim(\ker(\delta^{i+1})\cap\spann\{v^{\mu,i+1}_{\nu}|\mu\in I_{i+1}\})=1$. Since 
\end{proof}

The linear resolution given in Theorem \ref{koszthm} of course proves constructively the following corollary. 

\begin{mycor}
\label{koszcor1}
The category $\C\cat$ is Koszul.
\end{mycor}
Also using Corollary \ref{greqcor} one obtains the following.
\begin{mycor}
\label{koszcor2}
The category $\C\inj$ is Koszul. 
\end{mycor}

%%%%%%%%%%%%%
%

%
%%%%%%%%%%%%%%

\end{document}